\documentclass{amsart}
\usepackage{amsmath,amsfonts,amsthm}
\usepackage{amssymb,latexsym}
\usepackage{graphics}
\usepackage[colorlinks]{hyperref}

\theoremstyle{plain}
\theoremstyle{definition}
\newtheorem{theorem}{Theorem}[section]
\newtheorem{lemma}[theorem]{Lemma}

\newtheorem{convention}[theorem]{Convention}

\theoremstyle{remark}

\numberwithin{equation}{section}
\newcommand{\SP}{\: \: \: \: \:}

\title[Proximal relations and generalized shifts]{On proximal relations in transformation semigroups arising from generalized shifts}
\author[F. Ayatollah Zadeh Shirazi, A. Fallahpour, M. R. Mardanbeigi, Z. Nili Ahmadabadi]{F. Ayatollah Zadeh Shirazi \\ A. Fallahpour \\  M. R. Mardanbeigi \\ Z. Nili Ahmadabadi}
\begin{document}
\begin{abstract}
\noindent   For a finite discrete topological space $X$
with at least two elements, a nonempty set $\Gamma$, and a map $\varphi:\Gamma\to\Gamma$,
$\sigma_\varphi:X^\Gamma\to X^\Gamma$ with $\sigma_\varphi((x_\alpha)_{\alpha\in\Gamma})=
(x_{\varphi(\alpha)})_{\alpha\in\Gamma}$ (for $(x_\alpha)_{\alpha\in\Gamma}\in X^\Gamma$) is
a generalized shift. In this text for $\mathcal{S}=\{\sigma_\psi:\psi\in\Gamma^\Gamma\}$
and $\mathcal{H}=\{\sigma_\psi:
\Gamma\mathop{\rightarrow}\limits^{\psi}\Gamma$ is bijective$\}$ we study proximal relations
of transformation semigroups $(\mathcal{S},X^\Gamma)$ and $(\mathcal{H},X^\Gamma)$.
Regarding proximal relation we prove:
\[P({\mathcal S},X^\Gamma)=\{((x_\alpha)_{\alpha\in\Gamma},(y_\alpha)_{\alpha\in\Gamma})
\in X^\Gamma\times X^\Gamma:
\exists\beta\in\Gamma\:(x_\beta=y_\beta)\}\]
and $P({\mathcal H},X^\Gamma)\subseteq \{((x_\alpha)_{\alpha\in\Gamma},(y_\alpha)_{\alpha\in\Gamma})
\in X^\Gamma\times X^\Gamma:
\{\beta\in\Gamma:x_\beta=y_\beta\}$ is infinite~$\}\cup\{
(x,x):x\in \mathcal{X}\}$.
\\
Moreover, for infinite $\Gamma$, both transformation semigroups
$({\mathcal S},X^\Gamma)$ and $({\mathcal H},X^\Gamma)$ are regionally proximal,
i.e.,  $Q({\mathcal S},X^\Gamma)=Q({\mathcal H},X^\Gamma)=X^\Gamma
\times X^\Gamma$, also for sydetically proximal relation we have
$L({\mathcal H},X^\Gamma)=\{((x_\alpha)_{\alpha\in\Gamma},(y_\alpha)_{\alpha\in\Gamma})
\in X^\Gamma\times X^\Gamma:
\{\gamma\in\Gamma:x_\gamma\neq y_\gamma\}$ is finite$\}$.
\end{abstract}
\maketitle
{\small \noindent {\bf AMS Subject Classification 2010:} 54H20
}
%_____________________________________________________________________

%%%%%%%%%%%%%%%%%%%%%%%%%%%%%%%%%%%%
\section{Preliminaries}
%%%%%%%%%%%%%%%%%%%%%%%%%%%%%%%%%%%%
\noindent By a {\it (left topological) transformation semigroup} $(S,Z,\pi)$
or simply $(S,Z)$ we mean a compact Hausdorff
topological space $Z$ (phase space), discrete topological semigroup $S$ 
(phase semigroup) with identity $e$ and continuous
map $\pi:S\times Z\to Z$ ($\pi(s,z)=sz, s\in S,z\in Z$) such that for all $z\in Z$
and $s,t\in S$ we have $ez=z$, $(st)z=s(tz)$. If
$S$ is a discrete topological group too, then we call the transformation semigroup $(S,Z)$,
a {\it transformation group}.
We say $(x,y)\in Z\times Z$
is a {\it proximal pair} of $(S,Z)$ if 
there exists a net $\{s_\lambda\}_{\lambda\in\Lambda}$ in $S$
with ${\displaystyle\lim_{\lambda\in\Lambda}s_\lambda x=\lim_{\lambda\in\Lambda}s_\lambda y}$.
We denote the collection of all proximal pairs of $(S,Z)$ by $P(S,Z)$ and call
it {\it proximal relation} on $(S,Z)$, for more details on proximal relations
we refer the interested reader to  \cite{brn} and \cite{glasner}. 
\\
In the transformation semigroup $(S,Z)$ we call $(x,y)\in Z\times Z$
a regionally proximal pair if there exists a net $\{(s_\lambda,x_\lambda,y_\lambda)\}_{
\lambda\in\Lambda}$ in $S\times Z\times Z$ such that
${\displaystyle\lim_{\lambda\in\Lambda}x_\lambda}=x$,
${\displaystyle\lim_{\lambda\in\Lambda}y_\lambda}=y$, and
${\displaystyle\lim_{\lambda\in\Lambda}s_\lambda x_\lambda=
    \lim_{\lambda\in\Lambda}s_\lambda y_\lambda}$.
We denote the collection of all regionally proximal pairs of $(S,Z)$ by $Q(S,Z)$ and call
it regionally proximal relation on $(S,Z)$. Obviously we have
$P(S,Z)\subseteq Q(S,Z)$. 
In the transformation group $(T,Z)$, by \cite{yu} we call $L(T,Z)=\{(x,y)\in Z\times Z:\overline{T(x,y)}
\subseteq P(T,Z)\}$ the syndetically proximal relation of $(T,Z)$
(for details on the interaction of $L(T,Z)$, $Q(T,Z)$ and $P(T,Z)$ with uniform structure of $Z$
see \cite{ellis,gerco,yu}).
%%%%%%%%%%%%%%%%%%%%%%%%%%%%%%%%%%%%
\subsection*{A collection of generalized shifts as phase semigroup}
\noindent For nonempty sets $X,\Gamma$ and self-map $\varphi:\Gamma\to\Gamma$
define the generalized shift 
$\sigma_\varphi:X^\Gamma\to X^\Gamma$ by
$\sigma_\varphi((x_\alpha)_{\alpha\in\Gamma})=(x_{\varphi(\alpha)})_{\alpha\in\Gamma}$
($(x_\alpha)_{\alpha\in\Gamma}\in X^\Gamma$). Generalized shifts have been introduced
for the first time in \cite{200}, 
in addition dynamical and non-dynamical properties of generalized shifts 
have been studied in several texts like \cite{dev} and \cite{gio}. It's well-known that if $X$ 
has a topological structure, then $\sigma_\varphi:X^\Gamma\to X^\Gamma$
is continuous (when $X^\Gamma$ equipped with product topology), in addition
If $X$ has at least two elements, then $\sigma_\varphi:X^\Gamma\to X^\Gamma$
is a homeomorphism if and only if $\varphi:\Gamma\to\Gamma$ is bijective.
%%%%%%%%%%%%%%%%%%%%%%%%%%%%%%%%%%%%
\begin{convention}
In this text suppose $X$ is a finite discrete topological space with at least
two elements, $\Gamma$ is a nonempty set, ${\mathcal X}:=X^\Gamma$, and:
\begin{itemize}
\item ${\mathcal S}:=\{\sigma_\varphi:\varphi\in\Gamma^\Gamma\}$,
    is the semigroup of generalized shifts on $X^\Gamma$,
\item ${\mathcal H}:=\{\sigma_\varphi:\varphi\in\Gamma^\Gamma$ and $\varphi:\Gamma\to
    \Gamma$ is bijective~$\}$, is the group of generalized shift
    homeomorphisms on $X^\Gamma$.
\end{itemize}
Equip $X^\Gamma$ with product (pointwise convergence) topology.
Now we may consider ${\mathcal S}$ (resp. ${\mathcal H}$) as a subsemigroup
(resp. subgroup) of continuous maps (resp. homeomorphisms) from
${\mathcal X}$ to itself, so ${\mathcal S}$ (resp. ${\mathcal H}$)
acts on ${\mathcal X}$ in a natural way.
\end{convention}
\noindent Our aim in this text is to study 
$P(T,{\mathcal X})$, $Q(T,{\mathcal X})$, and $L(T,{\mathcal X})$ for $T=\mathcal{H},\mathcal{S}$.
Readers interested in this subject may refer to \cite{mogh} too.
%%%%%%%%%%%%%%%%%%%%%%%%%%%%%%%%%%%%
\section{Proximal and regionally proximal relations of $({\mathcal S},{\mathcal X})$}
\noindent In this section we prove that 
\[P({\mathcal S},{\mathcal X})=\{((x_\alpha)_{\alpha\in\Gamma},(y_\alpha)_{\alpha\in\Gamma})
\in {\mathcal X}\times {\mathcal X}:
\exists\beta\in\Gamma\:(x_\beta=y_\beta)\}\]
and 
\[Q(\mathcal{S},\mathcal{X})=\left\{\begin{array}{lc} \mathcal{X}\times\mathcal{X}
& \Gamma {\rm \: is \: infinite}\:, \\
P({\mathcal S},{\mathcal X}) & \Gamma {\rm \: is \: finite}\:.\end{array}\right.\]
%%%%%%%%%%%%%%%%%%%%%%%%%%%%%%%%%%%%
\begin{theorem}\label{th}
$P({\mathcal S},{\mathcal X})=\{((x_\alpha)_{\alpha\in\Gamma},(y_\alpha)_{\alpha\in\Gamma})
\in {\mathcal X}\times {\mathcal X}:
\exists\beta\in\Gamma\:(x_\beta=y_\beta)\}$.
\end{theorem}
%%%%%%%%%%%%%%%%%%%%%%%%%%%%%%%%%%%%
\begin{proof}
First consider $\beta\in \Gamma$ and $(x_\alpha)_{\alpha\in\Gamma},(y_\alpha)_{\alpha\in\Gamma}
\in {\mathcal X}$ by $x_\beta=y_\beta$. Define $\psi:\Gamma\to\Gamma$ with
$\psi(\alpha)=\beta$ for all $\alpha\in\Gamma$. Then 
\[\sigma_\psi(
(x_\alpha)_{\alpha\in\Gamma})=(x_\beta)_{\alpha\in\Gamma}=
(y_\beta)_{\alpha\in\Gamma}=\sigma_\psi(
(y_\alpha)_{\alpha\in\Gamma})\]
and
$((x_\alpha)_{\alpha\in\Gamma},(y_\alpha)_{\alpha\in\Gamma})\in
P({\mathcal S},{\mathcal X})$.
\\
Conversely, suppose $((x_\alpha)_{\alpha\in\Gamma},(y_\alpha)_{\alpha\in\Gamma})\in
P({\mathcal S},{\mathcal X})$. There exists a net
$\{\sigma_{\varphi_\lambda}\}_{\lambda\in\Lambda}$ in ${\mathcal S}$ with
${\displaystyle\lim_{\lambda\in\Lambda}\sigma_{\varphi_\lambda}(
(x_\alpha)_{\alpha\in\Gamma})=
\lim_{\lambda\in\Lambda}\sigma_{\varphi_\lambda}(
(y_\alpha)_{\alpha\in\Gamma})}=:(z_\alpha)_{\alpha\in\Gamma}$.
Choose arbitrary $\theta\in\Gamma$, then
\[{\displaystyle\lim_{\lambda\in\Lambda}x_{\varphi_\lambda(\theta)}=
\lim_{\lambda\in\Lambda}y_{\varphi_\lambda(\theta)}}=z_\theta\] 
in $X$.
Since $X$ is discrete, there exists $\lambda_0\in\Lambda$ such that
$x_{\varphi_\lambda(\theta)}=y_{\varphi_\lambda(\theta)}=z_\theta$
for all $\lambda\geq\lambda_0$, in particular for $\beta=\varphi_{\lambda_0(\theta)}$
we have $x_\beta=y_\beta$.
\end{proof}
%%%%%%%%%%%%%%%%%%%%%%%%%%%%%%%%%%%%
%%%%%%%%%%%%%%%%%%%%%%%%%%%%%%%%%%%%
\begin{lemma}\label{lem70}
For infinite $\Gamma$ we have:
$Q(\mathcal{S},\mathcal{X})=Q(\mathcal{H},\mathcal{X})=\mathcal{X}\times\mathcal{X}$.
\end{lemma}
%%%%%%%%%%%%%%%%%%%%%%%%%%%%%%%%%%%%
\begin{proof}
Suppose $\Gamma$ is infinite, then there exits a bijection $\mu:\Gamma\times{\mathbb Z}\to\Gamma$, in particular $\{\mu(\{\alpha\}\times{\mathbb Z}):\alpha\in\Gamma\}$
is a partition of $\Gamma$ to its infinite countable subsets. Define bijection
$\varphi:\Gamma\to\Gamma$
by $\varphi(\mu(\alpha,n))=\mu(\alpha,n+1)$ for all $\alpha\in\Gamma$ and
$n\in\mathbb Z$. Consider $p\in X$ and
$(x_\alpha)_{\alpha\in\Gamma},(y_\alpha)_{\alpha\in\Gamma}
\in\mathcal{X}$. For all $n\geq1$  and $\alpha\in\Gamma$ let:
\[x_\alpha^n:=\left\{\begin{array}{lc} x_\alpha & \alpha=\mu(\beta,k)
	{\rm \: for \: some \:}\beta\in\Gamma{\rm \: and \:}k\leq n\:,\\
	p & {\rm otherwise}\:,\end{array}\right.\]
and
	\[y_\alpha^n:=\left\{\begin{array}{lc} y_\alpha & \alpha=\mu(\beta,k)
	{\rm \: for \: some \:}\beta\in\Gamma{\rm \: and \:}k\leq n\:,\\
	p & {\rm otherwise}\:,\end{array}\right.\]
then:
\[{\displaystyle\lim_{n\to+\infty}(x^n_\alpha)_{\alpha\in\Gamma}}=(x_\alpha)_{\alpha\in\Gamma}\:,\]
\[{\displaystyle\lim_{n\to+\infty}(y^n_\alpha)_{\alpha\in\Gamma}}=(y_\alpha)_{\alpha\in\Gamma}\:,\]
\[{\displaystyle\lim_{n\to\infty}\sigma_{\varphi^{2n}}((x^n_\alpha)_{\alpha\in\Gamma})}=
(p_\alpha)_{\alpha\in\Gamma}=
{\displaystyle\lim_{n\to+\infty}\sigma_{\varphi^{2n}}((y^n_\alpha)_{\alpha\in\Gamma})}\:.\]
By $\sigma_{\varphi^{2n}}\in\mathcal{H}$ for all $n\geq1$ and using the
above statements, we have $((x_\alpha)_{\alpha\in\Gamma},(y_\alpha)_{\alpha\in\Gamma})\in
Q(\mathcal{H},\mathcal{X})\subseteq Q(\mathcal{S},\mathcal{X})$.
\end{proof}
%%%%%%%%%%%%%%%%%%%%%%%%%%%%%%%%%%%%
\begin{lemma}\label{lem80}
For finite $\Gamma$ and any subsemigroup $\mathcal T$ of $\mathcal S$ we have 
$Q(\mathcal{T},\mathcal{X})=P(\mathcal{T},\mathcal{X})$.
\end{lemma}
%%%%%%%%%%%%%%%%%%%%%%%%%%%%%%%%%%%%
\begin{proof}
We must only prove $Q(\mathcal{T},\mathcal{X})\subseteq P(\mathcal{T},\mathcal{X})$.
Suppose $(x,y)\in Q(\mathcal{T},\mathcal{X})$, then there exists a net $\{(x_\lambda,y_\lambda,
t_\lambda)\}_{\lambda\in\Lambda}$ in $\mathcal{X}\times\mathcal{X}\times\mathcal{T}$ such that
${\displaystyle\lim_{\lambda\in\Lambda}x_\lambda}=x$, 
${\displaystyle\lim_{\lambda\in\Lambda}y_\lambda}=y$,
and ${\displaystyle\lim_{\lambda\in\Lambda}t_\lambda x_\lambda}=
{\displaystyle\lim_{\lambda\in\Lambda}t_\lambda y_\lambda}=:z$. Since 
$\mathcal{X}\times\mathcal{X}\times\mathcal{T}$ is finite, $\{(x_\lambda,y_\lambda,
t_\lambda)\}_{\lambda\in\Lambda}$ has a constant subnet like $\{(x_{\lambda_\mu},y_{\lambda_\mu},
t_{\lambda_\mu})\}_{\mu\in M}$, so there exists $t\in\mathcal T$ such that for all $\mu\in M$ we have 
$x=x_{\lambda_\mu}$, $y=y_{\lambda_\mu}$ and $t=t_{\lambda_\mu}$, therefore
$tx=ty(=z)$ and $(x,y)\in P(\mathcal{T},\mathcal{X})$.
\end{proof}
%%%%%%%%%%%%%%%%%%%%%%%%%%%%%%%%%%%%
\begin{theorem}
We have:
\[Q(\mathcal{S},\mathcal{X})=\left\{\begin{array}{lc} \mathcal{X}\times\mathcal{X}
& \Gamma {\rm \: is \: infinite}\:, \\
P({\mathcal S},{\mathcal X}) & \Gamma {\rm \: is \: finite}\:.\end{array}\right.\]
\end{theorem}
%%%%%%%%%%%%%%%%%%%%%%%%%%%%%%%%%%%%
\begin{proof}
Use Lemmas~\ref{lem70} and~\ref{lem80}.
\end{proof}
%%%%%%%%%%%%%%%%%%%%%%%%%%%%%%%%%%%%
\section{Proximal and regionally proximal relations of $({\mathcal H},{\mathcal X})$}
\noindent Note that for finite $\Gamma$, $\mathcal{H}$ is a
finite subset of homeomorphisms on $\mathcal{X}$
and $P({\mathcal H},{\mathcal X})=\{(x,x):x\in\mathcal{X}\}$,
also using Lemmas~\ref{lem70} and~\ref{lem80} we have:
\[Q(\mathcal{H},\mathcal{X})=\left\{\begin{array}{lc} \mathcal{X}\times\mathcal{X}
& \Gamma {\rm \: is \: infinite}\:, \\
P({\mathcal H},{\mathcal X})=\{(x,x):x\in\mathcal{X}\} & 
\Gamma {\rm \: is \: finite}\:.\end{array}\right.\]
In this section we show that:
{\small
\[\{((x_\alpha)_{\alpha\in\Gamma},(y_\alpha)_{\alpha\in\Gamma}):
\max({\rm card}(\{\beta\in\Gamma:x_\beta\neq y_\beta\}),\aleph_0)\leq\:{\rm card}(\{\beta\in\Gamma:x_\beta=y_\beta\})\}\]}
is a subset of $P(\mathcal{H},\mathcal{X})$, which is a subset of
\[\{((x_\alpha)_{\alpha\in\Gamma},(y_\alpha)_{\alpha\in\Gamma})
\in {\mathcal X}\times {\mathcal X}:
\{\beta\in\Gamma:x_\beta=y_\beta\}{\rm \: is \: infinite}\}\cup\{
(x,x):x\in \mathcal{X}\}\]
in its turn.
In particular, for countable $\Gamma$ we prove
\[P({\mathcal H},{\mathcal X})=
\{((x_\alpha)_{\alpha\in\Gamma},(y_\alpha)_{\alpha\in\Gamma})
\in {\mathcal X}\times {\mathcal X}:
\{\beta\in\Gamma:x_\beta=y_\beta\}{\rm \: is \: infinite}\}\cup 
\{(x,x):x\in\mathcal{X}\}\:.\]
%%%%%%%%%%%%%%%%%%%%%%%%%%%%%%%%%%%%
\begin{lemma}\label{lem10}
For infinite $\Gamma$, we have:
\begin{center}
$P({\mathcal H},{\mathcal X})\subseteq
\{((x_\alpha)_{\alpha\in\Gamma},(y_\alpha)_{\alpha\in\Gamma})
\in {\mathcal X}\times {\mathcal X}:
\{\beta\in\Gamma:x_\beta=y_\beta\}$ is infinite~$\}$.
\end{center}
\end{lemma}
%%%%%%%%%%%%%%%%%%%%%%%%%%%%%%%%%%%%
\begin{proof}
Consider $((x_\alpha)_{\alpha\in\Gamma},(y_\alpha)_{\alpha\in\Gamma})\in
P({\mathcal H},{\mathcal X})$, then there exists a net
$\{\sigma_{\varphi_\lambda}\}_{\lambda\in\Lambda}$ in ${\mathcal H}$ with
${\displaystyle\lim_{\lambda\in\Lambda}\sigma_{\varphi_\lambda}(
(x_\alpha)_{\alpha\in\Gamma})=
\lim_{\lambda\in\Lambda}\sigma_{\varphi_\lambda}(
(y_\alpha)_{\alpha\in\Gamma})}=:(z_\alpha)_{\alpha\in\Gamma}$. Choose distinct
\linebreak
$\theta_1,\ldots,\theta_n\in\Gamma$. For all $i\in\{1,\ldots,n\}$ we have
${\displaystyle\lim_{\lambda\in\Lambda}x_{\varphi_\lambda(\theta_i)}=
\lim_{\lambda\in\Lambda}y_{\varphi_\lambda(\theta_i)}}=z_{\theta_i}$ in $X$,
so there exists $\lambda_1,\ldots,\lambda_n\in\Lambda$ with
$x_{\varphi_\lambda(\theta_i)}=y_{\varphi_\lambda(\theta_i)}=z_{\theta_i}$
for all $\lambda\geq\lambda_i$. There exists $\mu\in\Lambda$ with
$\mu\geq\lambda_1,\ldots,\lambda_n$, thus
$x_{\varphi_\mu(\theta_i)}=y_{\varphi_\mu(\theta_i)}$ for $i=1,\ldots,n$. Since
$\varphi_\mu:\Gamma\to\Gamma$ is bijective and $\theta_1,\ldots,\theta_n$ are
pairwise distinct, $\{\varphi_\mu(\theta_1),\ldots,\varphi_\mu(\theta_n)\}$
has exactly $n$ elements and
$\{\varphi_\mu(\theta_1),\ldots,\varphi_\mu(\theta_n)\}
\subseteq\{\beta\in\Gamma:x_\beta=y_\beta\}$. Hence
$\{\beta\in\Gamma:x_\beta=y_\beta\}$ has at least $n$ elements
(for all $n\geq1$) and it is infinite.
\end{proof}
%%%%%%%%%%%%%%%%%%%%%%%%%%%%%%%%%%%%
\begin{theorem}
$P({\mathcal H},{\mathcal X})\subseteq
\{((x_\alpha)_{\alpha\in\Gamma},(y_\alpha)_{\alpha\in\Gamma})
\in {\mathcal X}\times {\mathcal X}:
\{\beta\in\Gamma:x_\beta=y_\beta\}$ is infinite~$\}\cup\{
(x,x):x\in \mathcal{X}\}$.
\end{theorem}
%%%%%%%%%%%%%%%%%%%%%%%%%%%%%%%%%%%%
\begin{proof}
Use Lemma~\ref{lem10} and the fact that for finite $\Gamma$, $\mathcal{H}$
is a finite subset of homeomorphisms on $\mathcal{X}$. So  for finite $\Gamma$ we have
$P(\mathcal{H},\mathcal{X})=\{(w,w):w\in\mathcal{X}\}$.
\end{proof}
%%%%%%%%%%%%%%%%%%%%%%%%%%%%%%%%%%%%
\begin{lemma}\label{lem20}
For infinite countable $\Gamma$,
$P({\mathcal H},{\mathcal X})=
\big\{((x_\alpha)_{\alpha\in\Gamma},(y_\alpha)_{\alpha\in\Gamma})
\in {\mathcal X}\times {\mathcal X}:
\{\beta\in\Gamma:x_\beta=y_\beta\}$ is infinite~$\big\}$.
\end{lemma}
%%%%%%%%%%%%%%%%%%%%%%%%%%%%%%%%%%%%
\begin{proof}
Using Lemma~\ref{lem10} we must only prove:
\begin{center}
$P({\mathcal H},{\mathcal X})\supseteq
\{((x_\alpha)_{\alpha\in\Gamma},(y_\alpha)_{\alpha\in\Gamma})
\in {\mathcal X}\times {\mathcal X}:
\{\beta\in\Gamma:x_\beta=y_\beta\}$ is infinite~$\}$. 
\end{center}
Consider
$(x_\alpha)_{\alpha\in\Gamma},(y_\alpha)_{\alpha\in\Gamma}\in{\mathcal X}$
with infinite set  $\{\beta\in\Gamma:x_\beta=y_\beta\}=\{\beta_1,\beta_2,\ldots\}$
and distinct $\beta_i$s. Also suppose $\Gamma=\{\alpha_1,\alpha_2,\ldots\}$
with distinct $\alpha_i$s. For all $n\geq1$ there exists bijection
$\varphi_n:\Gamma\to\Gamma$ with $\varphi_n(\alpha_i)=\beta_i$ for $i\in\{1,\ldots,n\}$.
Let $\alpha\in\Gamma$, there exists $i\geq1$ with $\alpha=\alpha_i$.
Since for all $n\geq i$ we have
$x_{\varphi_n(\alpha)}=x_{\varphi_n(\alpha_i)}=x_{\beta_i}=y_{\beta_i}
=y_{\varphi_n(\alpha_i)}=y_{\varphi_n(\alpha)}$, we have
${\displaystyle\lim_{n\to\infty}x_{\varphi_n(\alpha)}=\lim_{n\to\infty}y_{\varphi_n(\alpha)}}$.
Therefore
\[{\displaystyle\lim_{n\to\infty}\sigma_{\varphi_n}((x_\alpha)_{\alpha\in\Gamma})
=\lim_{n\to\infty}(x_{\varphi_n(\alpha)})_{\alpha\in\Gamma}=
\lim_{n\to\infty}(y_{\varphi_n(\alpha)})_{\alpha\in\Gamma}=
\lim_{n\to\infty}\sigma_{\varphi_n}((y_\alpha)_{\alpha\in\Gamma})}\:,\]
and $((x_\alpha)_{\alpha\in\Gamma},(y_\alpha)_{\alpha\in\Gamma})\in
P({\mathcal H},{\mathcal X})$.
\end{proof}
%%%%%%%%%%%%%%%%%%%%%%%%%%%%%%%%%%%%
\begin{theorem}\label{th10}
For countable $\Gamma$,
\begin{center}
$P({\mathcal H},{\mathcal X})=
\{((x_\alpha)_{\alpha\in\Gamma},(y_\alpha)_{\alpha\in\Gamma})
\in {\mathcal X}\times {\mathcal X}:
\{\beta\in\Gamma:x_\beta=y_\beta\}$ is infinite~$\}\cup 
\{(x,x):x\in\mathcal{X}\}$.
\end{center}
\end{theorem}
%%%%%%%%%%%%%%%%%%%%%%%%%%%%%%%%%%%%
\begin{proof}
First note that for finite $\Gamma$, $\mathcal{H}$
is finite and $P({\mathcal H},{\mathcal X})=\{(x,x):x\in\mathcal{X}\}$. Now use
Lemma~\ref{lem20}.
\end{proof}
%%%%%%%%%%%%%%%%%%%%%%%%%%%%%%%%%%%%
\begin{lemma}\label{salam}
For infinite $\Gamma$, we have:
\[\{((x_\alpha)_{\alpha\in\Gamma},(y_\alpha)_{\alpha\in\Gamma}):
{\rm card}(\{\beta\in\Gamma:x_\beta\neq y_\beta\})\leq\:{\rm card}(\{\beta\in\Gamma:x_\beta=y_\beta\})\}\subseteq P(\mathcal{H},\mathcal{X})\:.\]
In particular, 
\[\{((x_\alpha)_{\alpha\in\Gamma},(y_\alpha)_{\alpha\in\Gamma}):
\{\beta\in\Gamma:x_\beta\neq y_\beta\}{\rm \: is \: finite}\}
\subseteq P(\mathcal{H},\mathcal{X})\:.\]
\end{lemma}
%%%%%%%%%%%%%%%%%%%%%%%%%%%%%%%%%%%%
\begin{proof}
Suppose $\Gamma$ is infinite. For $(x_\alpha)_{\alpha\in\Gamma},(y_\alpha)_{\alpha\in\Gamma}\in\mathcal{X}$,
let:
\[A:=\{\alpha\in\Gamma:x_\alpha=y_\alpha\}\SP,\SP B:=\{\alpha\in\Gamma:x_\alpha\neq y_\alpha\}\]
with ${\rm card}(B)\leq{\rm card}(A)$. There exists a one to one map
$\lambda:B\to A$. By ${\rm card}(\Gamma)={\rm card}(A)+{\rm card}(B)$ and 
${\rm card}(B)\leq{\rm card}(A)$, $A$ is infinite. 
Since $A$ is infinite, we have ${\rm card}(A)={\rm card}(A)\aleph_0$ so there exists
a bijection $\varphi:A\times{\mathbb N}\to A$. For all $\theta\in A$ let $K_\theta=
\varphi(\{\theta\}\times{\mathbb N})\cup\lambda^{-1}(\theta)$. Thus
$K_\theta$s are disjoint infinite countable subsets of $\Gamma$, as a matter of fact
$\{K_\theta:\theta\in A\}$ is a partition of $\Gamma$ to some of its infinite countable subsets. 
For all $\theta\in A$, $\{\alpha\in K_\theta:x_\alpha=y_\alpha\}=
\varphi(\{\theta\}\times{\mathbb N})$ is infinite and $K_\theta$ is infinite countable. By
Lemma~\ref{lem20} there exists a sequence $\{\psi_n^\theta\}$ of permutations on
$K_\theta$ such that ${\displaystyle\lim_{n\to\infty}\sigma_{\psi_n^\theta}
(x_\alpha)_{\alpha\in K_\theta}}={\displaystyle\lim_{n\to\infty}\sigma_{\psi_n^\theta}
(y_\alpha)_{\alpha\in K_\theta}}$. For all $n\geq1$ let
$\psi_n={\displaystyle\bigcup_{\theta\in A}\psi_n^\theta}$, then
$\psi_n:\Gamma\to\Gamma$ is bijective and ${\displaystyle\lim_{n\to\infty}\sigma_{\psi_n}
(x_\alpha)_{\alpha\in\Gamma}}={\displaystyle\lim_{n\to\infty}\sigma_{\psi_n}
(y_\alpha)_{\alpha\in \Gamma}}$, which completes the proof.
\end{proof}
%%%%%%%%%%%%%%%%%%%%%%%%%%%%%%%%%%%%
\begin{theorem}
The collection
$\{((x_\alpha)_{\alpha\in\Gamma},(y_\alpha)_{\alpha\in\Gamma}):
\max({\rm card}(\{\beta\in\Gamma:x_\beta\neq y_\beta\}),\aleph_0)\leq\:{\rm card}(\{\beta\in\Gamma:x_\beta=y_\beta\})\}$ is a subset of 
$P({\mathcal H},{\mathcal X})$.
\end{theorem}
%%%%%%%%%%%%%%%%%%%%%%%%%%%%%%%%%%%%
\begin{proof}
If $\Gamma$ is finite, then 
$\{((x_\alpha)_{\alpha\in\Gamma},(y_\alpha)_{\alpha\in\Gamma}):
\max({\rm card}(\{\beta\in\Gamma:x_\beta\neq y_\beta\}),\aleph_0)\leq\:{\rm card}(\{\beta\in\Gamma:x_\beta=y_\beta\})\}=\varnothing$. Use Lemma~\ref{salam} to complete the proof. 
\end{proof}
%%%%%%%%%%%%%%%%%%%%%%%%%%%%%%%%%%%%
%%%%%%%%%%%%%%%%%%%%%%%%%%%%%%%%%%%%
\section{Syndetically proximal relations of  $(\mathcal{H},\mathcal{X})$}
\noindent In this section we prove:
{\small
\[L(\mathcal{H},\mathcal{X})=\left\{\begin{array}{lc}
\{((x_\alpha)_{\alpha\in\Gamma},(y_\alpha)_{\alpha\in\Gamma})
\in \mathcal{X}\times\mathcal{X}:
\{\gamma\in\Gamma:x_\gamma\neq y_\gamma\}{\rm \: is \: finite}\} & \Gamma{\rm\: is \: infinite\:,} \\
\{(x,x):x\in \mathcal{X}\} & \Gamma{\rm\: is \: finite\:.}\end{array}\right.\]
}
%%%%%%%%%%%%%%%%%%%%%%%%%%%%%%%%%%%%
%%%%%%%%%%%%%%%%%%%%%%%%%%%%%%%%%%%%
\begin{lemma}\label{lem60}
For $(x_\alpha)_{\alpha\in\Gamma},(y_\alpha)_{\alpha\in\Gamma},
(u_\alpha)_{\alpha\in\Gamma}\in\mathcal{X}$,
and $p,q\in X$ let:
 \[z_\alpha:=\left\{\begin{array}{lc} q & x_\alpha\neq y_\alpha\:, \\ u_\alpha & x_\alpha= y_\alpha\:,\end{array}\right.\SP{\rm and}\SP 
 w_\alpha:=\left\{\begin{array}{lc} p & x_\alpha\neq y_\alpha\:, \\ u_\alpha & x_\alpha= y_\alpha\:.\end{array}\right.\]
 We have:
\begin{itemize}
\item[1.] if $((x_\alpha)_{\alpha\in\Gamma},(y_\alpha)_{\alpha\in\Gamma})\in P(
\mathcal{H},\mathcal{X})$, then 
$((z_\alpha)_{\alpha\in\Gamma},(w_\alpha)_{\alpha\in\Gamma})\in P(
\mathcal{H},\mathcal{X})$,
\item[2.] if $((x_\alpha)_{\alpha\in\Gamma},(y_\alpha)_{\alpha\in\Gamma})\in L(
\mathcal{H},\mathcal{X})$, then 
$((z_\alpha)_{\alpha\in\Gamma},(w_\alpha)_{\alpha\in\Gamma})\in L(
\mathcal{H},\mathcal{X})$.
\end{itemize}
\end{lemma}
%%%%%%%%%%%%%%%%%%%%%%%%%%%%%%%%%%%%
\begin{proof}
1) Suppose $((x_\alpha)_{\alpha\in\Gamma},(y_\alpha)_{\alpha\in\Gamma})\in P(
\mathcal{H},\mathcal{X})$, then there exists a net $\{\sigma_{\varphi_\lambda}\}_{\lambda\in 
\Lambda}$ in $\mathcal H$ such that ${\displaystyle\lim_{\lambda\in\Lambda}
\sigma_{\varphi_\lambda}((x_\alpha)_{\alpha\in\Gamma})}=
{\displaystyle\lim_{\lambda\in\Lambda}
\sigma_{\varphi_\lambda}((y_\alpha)_{\alpha\in\Gamma})}$. Thus
${\displaystyle\lim_{\lambda\in\Lambda}
((x_{\varphi_\lambda(\alpha)})_{\alpha\in\Gamma})}={\displaystyle\lim_{\lambda\in\Lambda}
((y_{\varphi_\lambda(\alpha)})_{\alpha\in\Gamma})}$, 
i.e. for all $\alpha\in\Gamma$
there exists $\kappa_\alpha\in\Lambda$ such that:
\[\forall\lambda\geq\kappa_\alpha\:(x_{\varphi_\lambda(\alpha)}=y_{\varphi_\lambda(\alpha)})\:.\]
Hence, for all $\lambda\geq\kappa_\alpha$ we have $z_{\varphi_\lambda(\alpha)}=u_{\varphi_\lambda(\alpha)}=w_{\varphi_\lambda(\alpha)}$.
On the other hand the net $\{(u_{\varphi_\lambda(\alpha)})_{\alpha\in\Gamma}\}_{\lambda\in 
\Lambda}$ has a convergent subnet like 
$\{(u_{\varphi_{\lambda_\theta}(\alpha)})_{\alpha\in\Gamma}\}_{\theta\in T}$
to a point of $\mathcal{X}$, say $(v_\alpha)_{\alpha\in\Gamma}$, since $\mathcal{X}$
is compact. For all $\alpha\in\Gamma$ there exists $\theta_\alpha\in T$ such that
$\lambda_{\theta_\alpha}\geq \kappa_\alpha$, and moreover
\[\forall\theta\geq\theta_\alpha\:(u_{\varphi_{\lambda_\theta}(\alpha)}=v_\alpha)\:.\]
Note that for all $\theta\geq\theta_\alpha$
we have $\lambda_\theta\geq\kappa_\alpha$, leads us to:
\[\forall\theta\geq\theta_\alpha\:
(z_{\varphi_{\lambda_\theta}(\alpha)}=v_\alpha=w_{\varphi_{\lambda_\theta}(\alpha)})\:.\]
Hence
${\displaystyle\lim_{\theta \in T}
\sigma_{\varphi_{\lambda_\theta}}((z_\alpha)_{\alpha\in\Gamma})}=
{\displaystyle\lim_{\theta\in T}
\sigma_{\varphi_{\lambda_\theta}}((w_\alpha)_{\alpha\in\Gamma})}$
and $((z_\alpha)_{\alpha\in\Gamma},(w_\alpha)_{\alpha\in\Gamma})\in P(
\mathcal{H},\mathcal{X})$.
\\
2) Now suppose  $((x_\alpha)_{\alpha\in\Gamma},(y_\alpha)_{\alpha\in\Gamma})\in L(
\mathcal{H},\mathcal{X})$ and
$((s_\alpha)_{\alpha\in\Gamma},(t_\alpha)_{\alpha\in\Gamma})$ is an
element of $\overline{\mathcal{H} 
((z_\alpha)_{\alpha\in\Gamma},(w_\alpha)_{\alpha\in\Gamma})}$.
There exists a net $\{\sigma_{\varphi_\lambda}\}_{\lambda\in \Lambda}$ in $\mathcal H$,
with 
\[((s_\alpha)_{\alpha\in\Gamma},(t_\alpha)_{\alpha\in\Gamma})={\displaystyle\lim_{
\lambda\in\Lambda}\sigma_{\varphi_\lambda}(
(z_\alpha)_{\alpha\in\Gamma},(w_\alpha)_{\alpha\in\Gamma})}={\displaystyle\lim_{
\lambda\in\Lambda}((z_{\varphi_\lambda(\alpha)})_{\alpha\in\Gamma}
,(w_{\varphi_\lambda(\alpha)})_{\alpha\in\Gamma})}\:.\]
On the other hand the net $\{((x_{\varphi_\lambda(\alpha)})_{\alpha\in\Gamma}
,(y_{\varphi_\lambda(\alpha)})_{\alpha\in\Gamma})\}_{\lambda\in\Lambda}$
has a convergent subnet in compact space $\mathcal{X}\times\mathcal{X}$, without
loss of generality we may suppose \linebreak
$\{((x_{\varphi_\lambda(\alpha)})_{\alpha\in\Gamma}
,(y_{\varphi_\lambda(\alpha)})_{\alpha\in\Gamma})\}_{\lambda\in\Lambda}$ itself
converges to a point of $\mathcal{X}\times\mathcal{X}$ like
\linebreak
$((m_\alpha)_{\alpha\in\Gamma},(n_\alpha)_{\alpha\in\Gamma})$.
Hence 
$((m_\alpha)_{\alpha\in\Gamma},(n_\alpha)_{\alpha\in\Gamma})\in
\overline{\mathcal{H} 
((x_\alpha)_{\alpha\in\Gamma},(y_\alpha)_{\alpha\in\Gamma})}\subseteq P
(\mathcal{H},\mathcal{X})$. Now for $\alpha\in\Gamma$  there exists $\kappa\in\Lambda$
such that:
\[\forall\lambda\geq\kappa\:((m_\alpha,n_\alpha)=(x_{\varphi_\lambda(\alpha)},y_{\varphi_\lambda(\alpha)}))\:.\]
Hence we have:
\begin{eqnarray*}
m_\alpha\neq n_\alpha & \Rightarrow & 
	(\forall\lambda\geq\kappa\:(x_{\varphi_\lambda(\alpha)}\neq y_{\varphi_\lambda(\alpha)})) \\
& \Rightarrow & (\forall\lambda\geq\kappa\:(z_{\varphi_\lambda(\alpha)}=q \wedge
	w_{\varphi_\lambda(\alpha)}=p)) \\
& \Rightarrow & {\displaystyle\lim_{\lambda\in\Lambda}z_{\varphi_\lambda(\alpha)}=q}
	\wedge {\displaystyle\lim_{\lambda\in\Lambda}w_{\varphi_\lambda(\alpha)}=p} \\
& \Rightarrow & (s_\alpha,t_\alpha)=(q,p)
\end{eqnarray*}
and:
\begin{eqnarray*}
m_\alpha= n_\alpha & \Rightarrow & 
	(\forall\lambda\geq\kappa\:(x_{\varphi_\lambda(\alpha)}= y_{\varphi_\lambda(\alpha)})) \\
& \Rightarrow & (\forall\lambda\geq\kappa\:(z_{\varphi_\lambda(\alpha)}=
	w_{\varphi_\lambda(\alpha)})) \\
& \Rightarrow & s_\alpha={\displaystyle\lim_{\lambda\in\Lambda}z_{\varphi_\lambda(\alpha)}}
	={\displaystyle\lim_{\lambda\in\Lambda}w_{\varphi_\lambda(\alpha)}}=t_\alpha \\
& \Rightarrow & s_\alpha=t_\alpha
\end{eqnarray*}
Hence for $(v_\alpha)_{\alpha\in\Gamma}:=(s_\alpha)_{\alpha\in\Gamma}$ we have:
\[s_\alpha=\left\{\begin{array}{lc} q & m_\alpha\neq n_\alpha \: ,\\
	v_\alpha & m_\alpha= n_\alpha \:, \end{array}\right.\SP{\rm and}\SP
	t_\alpha=\left\{\begin{array}{lc} p & m_\alpha\neq n_\alpha \: ,\\
	v_\alpha & m_\alpha= n_\alpha \:. \end{array}\right.\tag{*}\]
Using (1), $((m_\alpha)_{\alpha\in\Gamma},(n_\alpha)_{\alpha\in\Gamma})\in P
(\mathcal{H},\mathcal{X})$ and (*) we have
$((s_\alpha)_{\alpha\in\Gamma},(t_\alpha)_{\alpha\in\Gamma})\in P
(\mathcal{H},\mathcal{X})$, which completes the proof.
\end{proof}
%%%%%%%%%%%%%%%%%%%%%%%%%%%%%%%%%%%%
\begin{lemma}\label{lem40}
We have:
\[L(\mathcal{H},\mathcal{X})\subseteq\{((x_\alpha)_{\alpha\in\Gamma},(y_\alpha)_{\alpha\in\Gamma})
\in \mathcal{X}\times\mathcal{X}:
\{\gamma\in\Gamma:x_\gamma\neq y_\gamma\}{\rm \: is \: finite}\}
\:.\]
\end{lemma}%%%%%%%%%%%%%%%%%%%%%%%%%%%%%%%%%%%%
\begin{proof}
Consider $(x_\alpha)_{\alpha\in\Gamma},(y_\alpha)_{\alpha\in\Gamma}\in\mathcal{X}$
such that $B:=\{\alpha\in\Gamma:x_\alpha\neq y_\alpha\}$ is infinite. Choose distinct
$p,q\in X$ and let:
\[z_\alpha:=\left\{\begin{array}{lc} q & \alpha\in B\:, \\ p & \alpha\notin B\:.\end{array}\right.\]
By Lemma~\ref{lem60}, if $((x_\alpha)_{\alpha\in\Gamma},(y_\alpha)_{\alpha\in\Gamma})\in L(
\mathcal{H},\mathcal{X})$, then 
$((z_\alpha)_{\alpha\in\Gamma},(p)_{\alpha\in\Gamma})\in L(\mathcal{H},\mathcal{X})$.
We show $((q)_{\alpha\in\Gamma},(p)_{\alpha\in\Gamma})\in\overline{\mathcal{H}
((z_\alpha)_{\alpha\in\Gamma},(p)_{\alpha\in\Gamma})}$. Suppose $U$ is an open
neighbourhood of $((q)_{\alpha\in\Gamma},(p)_{\alpha\in\Gamma})$, then there exists
distinct $\alpha_1,\ldots,\alpha_n\in\Gamma$ such that for:
\[V_\alpha=\left\{\begin{array}{lc} \{q\} & \alpha=\alpha_1,\ldots,\alpha_n\:, \\ X &
\alpha\neq \alpha_1,\ldots,\alpha_n\:, \end{array}\right.\SP {\rm and}\SP
W_\alpha=\{p\}\:(\forall\alpha\in\Gamma)\:,\]
we have ${\displaystyle\prod_{\alpha\in\Gamma}V_\alpha}\times
{\displaystyle\prod_{\alpha\in\Gamma}W_\alpha}\subseteq U$.
Since $B$ is infinite, we could choose distinct $\beta_1,\ldots,\beta_n\in B$ 
such that $\{\alpha_1,\ldots,\alpha_n\}\cap\{\beta_1,\ldots,\beta_n\}=\varnothing$. 
Define $\psi:\Gamma\to\Gamma$ by
\[\psi(\alpha):=\left\{\begin{array}{lc} \alpha_i & \alpha=\beta_i,i=1,\ldots,n\:, \\
\beta_i & \alpha=\alpha_i,i=1,\ldots,n\:, \\ \alpha & {\rm otherwise\:,} \end{array}\right.\]
then $\psi:\Gamma\to\Gamma$ is bijective, $\sigma_\psi\in{\mathcal{H}}$ 
and 
\[\sigma_{\psi}((z_\alpha)_{\alpha\in\Gamma},
(p)_{\alpha\in\Gamma})=(\sigma_\psi((z_\alpha)_{\alpha\in\Gamma}),
\sigma_\psi((p)_{\alpha\in\Gamma}))=((z_{\psi(\alpha)})_{\alpha\in\Gamma},
(p)_{\alpha\in\Gamma})\in U\:.\]
Hence 
$((q)_{\alpha\in\Gamma},(p)_{\alpha\in\Gamma})\in\overline{\mathcal{H}
((z_\alpha)_{\alpha\in\Gamma},(p)_{\alpha\in\Gamma})}$. Since 
$((q)_{\alpha\in\Gamma},(p)_{\alpha\in\Gamma})\notin P(\mathcal{H},\mathcal{X})$,
we have $((z_\alpha)_{\alpha\in\Gamma},(p)_{\alpha\in\Gamma})
\notin L(\mathcal{H},\mathcal{X})$, which leads to  
 $((x_\alpha)_{\alpha\in\Gamma},(y_\alpha)_{\alpha\in\Gamma})\notin L(
\mathcal{H},\mathcal{X})$ and completes the proof.
\end{proof}
\noindent The proof of the following lemma is similar to that of 
Lemma~\ref{lem10}.
%%%%%%%%%%%%%%%%%%%%%%%%%%%%%%%%%%%%
\begin{lemma}\label{lem30}
For $((x_\alpha)_{\alpha\in\Gamma},(y_\alpha)_{\alpha\in\Gamma})\in 
\mathcal{X}\times\mathcal{X}$  if
$\{\alpha\in\Gamma:x_\alpha\neq y_\alpha\}$ is finite and
$((z_\alpha)_{\alpha\in\Gamma},(w_\alpha)_{\alpha\in\Gamma})\in\overline{\mathcal{H}
((x_\alpha)_{\alpha\in\Gamma},(y_\alpha)_{\alpha\in\Gamma})}$, then 
$\{\alpha\in\Gamma:z_\alpha\neq w_\alpha\}$ is finite satisfying
${\rm card}(\{\alpha\in\Gamma:z_\alpha\neq w_\alpha\})\leq {\rm card}
(\{\alpha\in\Gamma:x_\alpha\neq y_\alpha\})$.
\end{lemma}
%%%%%%%%%%%%%%%%%%%%%%%%%%%%%%%%%%%%
\begin{proof}
For $n\geq1$, if there exists distinct $\alpha_1,\ldots,\alpha_n\in\Gamma$ with
$z_{\alpha_i}\neq w_{\alpha_i}$ for $i=1,\ldots,n$, then let:
\[U_\alpha:=\left\{\begin{array}{lc} \{z_\alpha\} & \alpha=\alpha_1,\ldots,\alpha_n\:, \\
X & \alpha\neq\alpha_1,\ldots,\alpha_n\:,\end{array}\right.\SP{\rm and}\SP
V_\alpha:=\left\{\begin{array}{lc} \{w_\alpha\} & \alpha=\alpha_1,\ldots,\alpha_n\:, \\
X & \alpha\neq\alpha_1,\ldots,\alpha_n\:.\end{array}\right.\]
Thus $U:={\displaystyle\prod_{\alpha\in\Gamma}U_\alpha}\times
{\displaystyle\prod_{\alpha\in\Gamma}V_\alpha}$ is an open neighbourhood of
$((z_\alpha)_{\alpha\in\Gamma},(w_\alpha)_{\alpha\in\Gamma})$, and there exists
bijection $\varphi:\Gamma\to\Gamma$ with 
\[(\sigma_\varphi((x_\alpha)_{\alpha\in\Gamma}),\sigma_\varphi((y_\alpha)_{\alpha\in\Gamma}))=
((x_{\varphi(\alpha)})_{\alpha\in\Gamma},(y_{\varphi(\alpha)})_{\alpha\in\Gamma})\in U\:.\]
Hence $x_{\varphi(\alpha_i)}=z_{\alpha_i}$ and $y_{\varphi(\alpha_i)}=w_{\alpha_i}$
for all $i=1,\ldots,n$. Therefore $x_{\varphi(\alpha_i)}\neq y_{\varphi(\alpha_i)}$
for all $i=1,\ldots,n$, which leads to $\{\varphi(\alpha_1),\ldots,\varphi(\alpha_n)\}
\subseteq \{\alpha\in\Gamma:x_\alpha\neq y_\alpha\}$, so
$n={\rm card}(\{\varphi(\alpha_1),\ldots,\varphi(\alpha_n)\})\leq{\rm card}(
\{\alpha\in\Gamma:x_\alpha\neq y_\alpha\})$ 
(note that $\varphi$ is one to one), which leads to the desired result.
\end{proof}
%%%%%%%%%%%%%%%%%%%%%%%%%%%%%%%%%%%%
\begin{lemma}\label{lem50}
For infinite $\Gamma$ we have:
\[L(\mathcal{H},\mathcal{X})\supseteq\{((x_\alpha)_{\alpha\in\Gamma},(y_\alpha)_{\alpha\in\Gamma})
\in \mathcal{X}\times\mathcal{X}:
\{\gamma\in\Gamma:x_\gamma\neq y_\gamma\}{\rm \: is \: finite}\}
\:.\]
\end{lemma}
%%%%%%%%%%%%%%%%%%%%%%%%%%%%%%%%%%%%
\begin{proof}
Use Lemmas~\ref{lem30} and \ref{salam}.
\end{proof}
%%%%%%%%%%%%%%%%%%%%%%%%%%%%%%%%%%%%
\begin{theorem}
We have:
{\small
\[L(\mathcal{H},\mathcal{X})=\left\{\begin{array}{lc}
\{((x_\alpha)_{\alpha\in\Gamma},(y_\alpha)_{\alpha\in\Gamma})
\in \mathcal{X}\times\mathcal{X}:
\{\gamma\in\Gamma:x_\gamma\neq y_\gamma\}{\rm \: is \: finite}\} & \Gamma{\rm\: is \: infinite\:,} \\
\{(x,x):x\in \mathcal{X}\} & \Gamma{\rm\: is \: finite\:.}\end{array}\right.\]
}
\end{theorem}
%%%%%%%%%%%%%%%%%%%%%%%%%%%%%%%%%%%%
\begin{proof}
For infinite $\Gamma$ use Lemmas~\ref{lem40} and~\ref{lem50}, also for finite
$\Gamma$ note that $P(\mathcal{H},\mathcal{X})=\{(x,x):x\in \mathcal{X}\}$.
\end{proof}
%%%%%%%%%%%%%%%%%%%%%%%%%%%%%%%%%%%%
\section{More details}
\noindent In transformation semigroup $(S,W)$ we say a nonempty subset $D$ of $W$ is
invariant if $SD:=\{sw:s\in S,w\in D\}\subseteq W$. For closed invariant subset $D$ of
$W$ we may consider action of $S$ on $D$ in a natural way. For closed invariant subset $D$ of
$W$ one may verify easily, $P(S,D)\subseteq P(S,W)$, $Q(S,D)\subseteq Q(S,W)$,
and $L(S,D)\subseteq L(S,W)$. 
Suppose $Z$ is a compact Hausdorff topological space with at least two elements,
by Tychonoff's theorem $Z^\Gamma$ is also compact Hausdorff. Again for $\varphi:\Gamma\to\Gamma$
one may consider $\sigma_\varphi:Z^\Gamma\to Z^\Gamma$ ($\sigma_\varphi((z_\alpha)_{\alpha\in\Gamma})=
(z_{\varphi(\alpha)})_{\alpha\in\Gamma}$), also
${\mathcal S}:=\{\sigma_\varphi:Z^\Gamma\to Z^\Gamma|\varphi\in\Gamma^\Gamma\}$, and
${\mathcal H}:=\{\sigma_\varphi:Z^\Gamma\to Z^\Gamma| \varphi\in\Gamma^\Gamma$ and $\varphi:\Gamma\to
    \Gamma$ is bijective~$\}$. 
Then for each finite nonenpty subset $A$ of $Z$, $A^\Gamma$ is a closed invariant subset of $(\mathcal{S},Z^\Gamma)$
(resp.  $(\mathcal{H},Z^\Gamma)$)
and $A$ is a discrete (and finite) subset of $Z$. But using previous sections we know about
$P(T,A^\Gamma)$, $Q(T,A^\Gamma)$, and $L(T,A^\Gamma)$ for $T=\mathcal{H},\mathcal{S}$.
Hence for $T=\mathcal{H},\mathcal{S}$ by:
\[\bigcup\{P(T,A^\Gamma):A{\rm \: \: is \: \: a \: finite \: subset \: of\:} Z\}\subseteq P(T,Z^\Gamma)\:,\]
\[\bigcup\{Q(T,A^\Gamma):A{\rm \: \: is \: \: a \: finite \: subset \: of\:} Z\}\subseteq Q(T,Z^\Gamma)\:,\]
\[\bigcup\{L(T,A^\Gamma):A{\rm \: \: is \: \: a \: finite \: subset \: of\:} Z\}\subseteq L(T,Z^\Gamma)\:,\]
we will have more data about $P(T,Z^\Gamma),Q(T,Z^\Gamma),L(T,Z^\Gamma)$.
%%%%%%%%%%%%%%%%%%%%%%%%%%%%%%%%%%%%
%%%%%%%%%%%%%%%%%%%%%%%%%%%%%%%%%%%%
%%%%%%%%%%%%%%%%%%%%%%%%%%%%%%%%%%%%
%%%%%%%%%%%%%%%%%%%%%%%%%%%%%%%%%%%%%%%%%

{ \small {\bf Fatemah Ayatollah Zadeh Shirazi}, 
Faculty of Mathematics Statistics and Computer Science,
College of Science, University of Tehran,
Enghelab Ave., Tehran, Iran (e-mail: fatemah@khayam.ut.ac.ir)}

{ \small {\bf Amir Fallahpour}, 
Faculty of Mathematics Statistics and Computer Science,
College of Science, University of Tehran,
Enghelab Ave., Tehran, Iran (e-mail: amir.falah90@yahoo.com)}

{ \small {\bf Mohammad Reza Mardanbeigi}, 
Islamic Azad University, Science and Research Branch
Tehran, Iran (e-mail: mrmardanbeigi@srbiau.ac.ir)} 

{ \small {\bf Zahra Nili Ahmadabadi}, 
Islamic Azad University, Science and Research Branch
Tehran, Iran (e-mail: zahra.nili.a@gmail.com)}
\end{document}